\documentclass[a4paper,10pt]{article}
\usepackage{mathtext}
\usepackage[T1,T2A]{fontenc}
\usepackage[cp1251]{inputenc}
\usepackage[english]{babel}
\usepackage{amsmath}
\usepackage{amsfonts}
\usepackage{amssymb}
\usepackage{mathrsfs}
\usepackage{amsthm}
\usepackage{enumerate}
\usepackage{graphicx}

\usepackage{color}
\usepackage{euscript}

\usepackage{cite}

\newtheorem{Le}{Lemma}[section]
\newtheorem{Def}[Le]{Definition}
\newtheorem{St}[Le]{Proposition}
\newtheorem{Th}{Theorem}[section]
\newtheorem{Cor}[Le]{Corollary}
\newtheorem{Rem}[Le]{Remark}

\numberwithin{equation}{section}

\newcommand{\R}{\mathbb{R}}

\newcommand{\N}{\mathbb{N}}
\newcommand{\Z}{\mathbb{Z}}

\DeclareMathOperator{\supp}{supp}

\newcommand{\eps}{\varepsilon}

\newcommand{\eq}[1]{\begin{equation}{#1}\end{equation}}
\newcommand{\mlt}[1]{\begin{multline}{#1}\end{multline}}
\newcommand{\alg}[1]{\begin{align}{#1}\end{align}}

\newcommand{\set}[2]{\{{#1}\mid{#2}\}}
\newcommand{\Set}[2]{\Big\{{#1}\,\Big|\;{#2}\Big\}}
\newcommand{\scalprod}[2]{\langle{#1},{#2}\rangle}

\newcommand{\Lseqref}[1]{\stackrel{\scriptscriptstyle{\eqref{#1}}}{\lesssim}}

\newcommand{\LseqrefTwo}[2]{\stackrel{\scriptscriptstyle{\eqref{#1},\eqref{#2}}}{\lesssim}}

\DeclareMathOperator{\I}{I}

\newcommand{\E}{\mathrm{E}}

\DeclareMathOperator{\diam}{diam}
\newcommand{\D}{\mathbb{D}}
\DeclareMathOperator{\cl}{cl}

\newcommand{\norm}{\vec{\mathrm{n}}}

\newcommand{\B}{\mathfrak{B}}

\newcommand{\Me}{\mathcal{M}}
\newcommand{\MM}{\mathbb{M}}

\DeclareMathOperator{\M}{M}

\newcommand{\Eb}{\mathrm{E}^{\boldsymbol{b}}}

\DeclareMathOperator{\DN}{DN}
\newcommand{\DNo}{\DN_{\mathrm{out}}}
\newcommand{\DNi}{\DN_{\mathrm{in}}}

\title{Maz'ya's~$\Phi$-inequalities on domains}
\author{Dmitriy Stolyarov\thanks{Supported by Russian Science Foundation grant N 19-71-10023}}
\begin{document}
\maketitle
\begin{abstract}
We find necessary and sufficient conditions on the function~$\Phi$ for the inequality
$$\Big|\int_\Omega \Phi(K*f)\Big|\lesssim \|f\|_{L_1(\R^d)}^p$$
to be true. Here~$K$ is a positively homogeneous of order~$\alpha - d$, possibly vector valued, kernel,~$\Phi$ is a~$p$-homogeneous function, and~$p=d/(d-\alpha)$. The domain~$\Omega\subset \R^d$ is either bounded with~$C^{1,\beta}$ smooth boundary for some~$\beta > 0$ or a halfspace in~$\R^d$. As a corollary, we describe the positively homogeneous of order~$d/(d-1)$ functions~$\Phi\colon \R^d \to \R$ that are suitable for the bound
$$\Big|\int_\Omega \Phi(\nabla u)\Big|\lesssim \int_\Omega |\Delta u|.$$
\end{abstract}

\section{Introduction}
This paper contains development of the theory of Maz'ya's~$\Phi$-inequalities given in~\cite{Stolyarov2021bis} and~\cite{Stolyarov2021}. These inequalities might be thought of as corrections of the invalid endpoint Hardy--Littlewood--Sobolev inequality
\eq{\label{HLSp=1}
\|\I_\alpha [f]\|_{L_{p}(\R^d)} \lesssim \|f\|_{L_1(\R^d)},\qquad p = \frac{d}{d-\alpha}.
}
The symbol~$\I_\alpha$ denotes the Riesz potential of order~$\alpha$:
\eq{
\I_\alpha [f] = c_{d,\alpha}\int\limits_{\R^d}\frac{f(y)}{|x-y|^{d-\alpha}}\,dy,\qquad f\in C_0^\infty(\R^d).
}
For the classical Hardy--Littlewood--Sobolev inequality and its applications, see the original paper~\cite{Sobolev1938} or modern presentation in Subsection~$\mathrm{VIII}.4.2$ of~\cite{Stein1993}. The counterexample to~\eqref{HLSp=1} is provided by a sequence of functions~$f$ approximating a delta measure. Linear corrections of the Hardy--Littlewood--Sobolev inequality at the endpoint are sometimes called Bourgain--Brezis inequalities. Usually, one adds translation and dilation invariant constraints for the function~$f$ that exclude the delta measures. We refer the reader to the pioneering paper~\cite{BB2004}, to more recent studies~\cite{GRV2024},~\cite{Stolyarov2022},~\cite{VanSchaftingen2013}, and to the surveys~\cite{Spector2020} and~\cite{VanSchaftingen2014}.

In~\cite{Maz'ya2010}, Vladimir Maz'ya suggested a modification of the Hardy--Littlewood--Sobolev inequality for~$p=1$ that allowed substitution of a delta measure (see also Problem~$5.1$ in the problem book~\cite{Maz'ya2018} as well). He conjectured that the inequality
\eq{\label{MazyaOriginal}
\Big|\int\limits_{\R^d} \Phi(\nabla u(x))\,dx\Big|\lesssim \|\Delta u\|_{L_1(\R^d)}^{d/(d-1)}
} 
holds true for all smooth compactly supported functions~$u$ with a uniform constant, provided the positively~$d/(d-1)$ homogeneous function~$\Phi\colon \R^d \to \R$ satisfies the cancellation condition
\eq{\label{Maz'jaCancellation}
\int\limits_{S^{d-1}}\Phi(\zeta)\,d\sigma(\zeta) = 0.
}
By~$\sigma$ we denote the Hausdorff measure on the unit sphere. The necessity of this condition follows from substitution of~$u$ that mimics the fundamental solution of the Laplace equation, i.e.~$\Delta u$ mimics a delta measure (see Section~\ref{S2} below for details on such arguments). By a positively~$q$-homogeneous function we mean a function~$\Psi$ defined on a Euclidean space such that~$\Psi(\lambda y) = \lambda^q \Psi(y)$ for any~$y$ and any~$\lambda \in \R_+$. Note that the inequality~\eqref{MazyaOriginal} is non-linear; there is no hint to convexity in it as well.  The cases where~$p=2$ and~$\Phi$ is a quadratic form were considered in~\cite{MS2009}, with sharp constants established. Maz'ya's conjecture was proved in~\cite{Stolyarov2021bis}. We cite the main result of that paper.
\begin{Th}\label{OldTheorem}
Let~$d$ be a natural number and let~$\alpha \in (0,d)$. Assume the kernel~$K\colon \R^d \to \R^\ell$ is positively~$(\alpha - d)$-homogeneous and Lipschitz on the unit sphere. Assume the function~$\Phi\colon \R^\ell \to \R$ is positively~$p$-homogeneous,~$p = d/(d-\alpha)$, and Lipschitz on the unit sphere. The inequality
\eq{
\Big|\int\limits_{\R^d} \Phi\Big(K*f(x)\Big)\,dx\Big|\lesssim \|f\|_{L_1(\R^d)}^{p}
}
holds true for all smooth compactly supported functions~$f$ with  zero mean if and only if
\eq{\label{CancellationOld}
\int\limits_{S^{d-1}}\Phi(K(\zeta))\,d\sigma(\zeta) = 0\quad \text{and}\quad \int\limits_{S^{d-1}}\Phi(-K(\zeta))\,d\sigma(\zeta) = 0.
}
\end{Th}

Theorem~\ref{OldTheorem} implies Maz'ya's conjecture~\eqref{MazyaOriginal} via the representation
\eq{\label{FormulaForGradient}
\nabla u(x) = c_d\int\limits_{\R^d}\frac{(x-y)\Delta u(y)}{|x-y|^{d}}\,dy,\qquad u\in C_0^\infty(\R^d),
}
here~$c_d$ is a certain constant. The classical Sobolev inequalities are valid for functions on domains. Since the companion Bourgain--Brezis inequalities were adjusted to this setting in~\cite{GR2019} and~\cite{GRV2024}, it is also desirable to find analogs of Maz'ya's~$\Phi$ inequalities for functions on domains. Here the main results are. In these theorems, we assume that~$K\colon \R^d \to \R^\ell$ and~$\Phi\colon \R^\ell \to \R$ are Lipschitz on the unit sphere and positively~$(\alpha- d)$ and~$p$ homogeneous, respectively. We always have the homogeneity relation~$p = d/(d-\alpha)$. The number~$\beta$ belongs to~$(0,1)$.

\begin{Th}\label{BoundedDomainTheorem}
Let~$\Omega \subset \R^d$ be a bounded domain whose boundary is~$C^{1,\beta}$ smooth. The inequality
\eq{\label{BoundedInequality}
\Big|\int\limits_{\Omega} \Phi\Big(K*f(x)\Big)\,dx\Big|\lesssim \|f\|_{L_1(\R^d)}^{p}
}
holds true for any bounded compactly supported function~$f$ on~$\R^d$ if and only if for any~$\xi \in S^{d-1}$
\eq{\label{CancellationNew}
\int\limits_{\genfrac{}{}{0pt}{-2}{\zeta \in S^{d-1},}{\scalprod{\zeta}{\xi} > 0}}\Phi(K(\zeta))\,d\sigma(\zeta) = 0\quad \text{and}\quad \int\limits_{\genfrac{}{}{0pt}{-2}{\zeta \in S^{d-1},}{\scalprod{\zeta}{\xi} > 0}}\Phi(-K(\zeta))\,d\sigma(\zeta) = 0.
}
\end{Th}
We note that the new cancellation condition~\eqref{CancellationNew} implies~\eqref{CancellationOld}. The main difference of the~$\Phi$-inequalities from classical estimates in the spirit of Sobolev is that there is no monotonicity of the estimated quantity with respect to domain. In other words, it is completely unclear how to bound~$\int_\Omega \Phi(K*f)$ with~$\int_{\R^d}\Phi(K*f)$; seemingly, there are no such bounds. In the classical case, however,~$\int_\Omega |f|^p \leq \int_{\R^d}|f|^p$.

We have also considered the case where~$\Omega$ is a half-space. 
\begin{Th}\label{HalfSpaceTheorem}
Fix~$\xi\in S^{d-1}$. Let~$\Omega  = \set{x\in \R^d}{\scalprod{x}{\xi} > 0}$. The inequality
\eq{\label{HalfSpaceInequality}
\Big|\int\limits_{\Omega} \Phi\Big(K*f(x)\Big)\,dx\Big|\lesssim \|f\|_{L_1(\R^d)}^{p}
}
holds true for any bounded compactly supported function~$f$ on~$\R^d$ with zero mean if and only if~\eqref{CancellationOld} and~\eqref{CancellationNew} hold true (with the specific vector~$\xi$). 
\end{Th}
The requirement that the boundary of~$\Omega$ is~$C^{1,\beta}$ smooth in Theorem~\ref{BoundedDomainTheorem} means the following. In a neighborhood of any boundary point,~$\Omega$ coincides with the subgraph of a differentiable function whose gradient is~$\beta$-H\"older. In particular, if we choose the coordinates in the way that in a neighborhood of the origin
\eq{
\Omega = \set{x\in\R^d}{x_d \geq h(x_1,x_2,\ldots,x_{d-1})},\quad  h(0)=0, \text{ and } \nabla h(0) = 0,
} 
then
\eq{\label{Holder}
|h(y)| \leq C|y|^{1+\beta}
}
for sufficiently small~$y\in \R^{d-1}$ and an absolute constant~$C = C(\Omega)$.

The assumption that~$f$ has zero mean appearing in the Theorems~\ref{OldTheorem} and~\ref{HalfSpaceTheorem} is necessary, since without them~$K*f$ decays as~$|x|^{\alpha - d}$ at infinity, and there is no hope for the bound in question unless we regularize the integral on the left hand side of the inequality; this requirement might be thought of as a condition of compact support for functions in Sobolev inequalities.  The proofs of the two theorems above rely on the circle of ideas developed in~\cite{Stolyarov2021bis} and~\cite{Stolyarov2021} and form, in a sense, a natural continuation of that papers. We also present corollaries in the spirit of the original Maz'ya's setting~\eqref{MazyaOriginal}.

\begin{Cor}\label{HarmonicCorollary}
Let~$\Phi\colon \R^d \to \R$ be a positively~$d/(d-1)$-homogeneous function whose restriction to the unit sphere is Lipschitz. Assume for any~$\xi\in \R^d \setminus \{0\}$ the cancellation condition
\eq{
\int\limits_{\genfrac{}{}{0pt}{-2}{\zeta \in S^{d-1},}{\scalprod{\zeta}{\xi} > 0}}\Phi(\zeta)\, d\sigma(\zeta) = 0
}
holds true. Then, for any bounded domain~$\Omega$ with smooth boundary and any smooth function~$u\colon \bar\Omega \to \R$, the estimate
\eq{\label{eq114}
\Big|\int\limits_\Omega\Phi(\nabla u(x))\,dx\Big|\lesssim \|\Delta u\|_{L_1(\Omega)} + \Big\|\frac{\partial u}{\partial n}\Big\|_{L_1(\partial \Omega)}
}
is also true with a uniform constant.
\end{Cor}
We have used the notation~$\bar\Omega$ to denote the closure of~$\Omega$. The boundary~$\partial \Omega$ is equipped with the Hausdorff measure of dimension~$d-1$. The second summand on the right hand side of~\eqref{eq114} is unavoidable as it can be seen from considering harmonic functions~$u$.
\begin{Cor}\label{HarmonicHalfspaceCorollary}
Let~$\Phi\colon \R^d \to \R$ be a positively~$d/(d-1)$-homogeneous function whose restriction to the unit sphere is Lipschitz. Let~$\xi \in S^{d-1}$. Assume the cancellation condition
\eq{\label{SpecificCancellation}
\int\limits_{\genfrac{}{}{0pt}{-2}{\zeta \in S^{d-1},}{\scalprod{\zeta}{\xi} > 0}}\Phi(\zeta)\, d\sigma(\zeta) = 0
}
holds true. Then, the estimate
\eq{
\Big|\int\limits_\Omega\Phi(\nabla u(x))\,dx\Big|\lesssim \|\Delta u\|_{L_1(\Omega)} + \Big\|\frac{\partial u}{\partial n}\Big\|_{L_1(\partial \Omega)}
}
is also true for any compactly supported smooth function~$u\colon \Omega \to \R$ with a uniform constant, where~$\Omega$ is the halfspace~$\set{x\in \R^d}{\scalprod{x}{\xi} > 0}$.
\end{Cor}
\begin{Rem}
The cancellation conditions in the two corollaries above are not only sufficient, but also necessary. The proof of this assertion is similar to the proofs of necessity in Theorems~\ref{BoundedDomainTheorem} and~\ref{HalfSpaceTheorem} presented in Section~\ref{S2} below.
\end{Rem}
\begin{Rem}
Though the smoothness assumptions on the boundary of~$\Omega$ in Corollary~\ref{HarmonicCorollary} seem superfluous, our proof uses the theory of pseudodifferential operators, and so we do not see and immediate way to remove them.
\end{Rem}
While Corollary~\ref{HarmonicHalfspaceCorollary} follows from Theorem~\ref{HalfSpaceTheorem} almost immediately, the derivation of Corollary~\ref{HarmonicCorollary} from Theorem~\ref{BoundedDomainTheorem} requires some efforts. The strategy is to extend~$u$ to a compactly supported function on~$\R^d$ in such a way that the~$L_1$-norm of the Laplacian is controlled by the right hand side of~\eqref{eq114} and then apply Theorem~\ref{BoundedDomainTheorem} via~\eqref{FormulaForGradient}. A similar strategy was used in~\cite{GRV2024} in the context of classical Bourgain--Brezis inequalities. Here we need a specific extension theorem, see Proposition~\ref{ExtensionTheorem}.

In the forthcoming Section~\ref{S2}, we prove the necessity of cancellation conditions in Theorems~\ref{BoundedDomainTheorem} and~\ref{HalfSpaceTheorem}. Section~\ref{S3} is devoted to reduction of our inequalities to Theorem~\ref{BesovTheorem} below; the latter theorem might be informally thought of as the Besov space version of Maz'ya's inequalities. Section~\ref{S4} contains the proof of the latter theorem. The concluding Section~\ref{S5} finishes the proofs of Theorems~\ref{BoundedDomainTheorem} and~\ref{HalfSpaceTheorem} and provides derivation Corollaries~\ref{HarmonicCorollary} and~\ref{HarmonicHalfspaceCorollary}.

I wish to thank Vladimir Maz'ya for asking me the question that motivated the paper (the question was whether Corollary~\ref{HarmonicCorollary} holds true). I am also grateful to Alexander Nazarov, Mikhail Novikov, Bogdan Rai\c{t}\u{a}, and Alexander Tyulenev for discussions concerning extension of smooth functions (results around Proposition~\ref{ExtensionTheorem}).  

\section{Necessity}\label{S2}
We first verify the necessity of~\eqref{CancellationNew} in Theorem~\ref{BoundedDomainTheorem} and then explain the modifications needed to justify necessity in Theorem~\ref{HalfSpaceTheorem}. 

Let~$\xi\in S^{d-1}$. Choose a point~$z_\xi\in \partial \Omega$ such that the inward pointing normal vector to~$\partial \Omega$ at~$z_\xi$ equals~$\xi$. For example, we may choose~$z_\xi$ as follows: let it be the point~$z\in \cl\Omega$ such that~$\scalprod{z}{\xi}$ is the smallest among all~$z\in \Omega$. By rotating and shifting~$\Omega$, we may, without loss of generality, assume~$\xi = (0,0,\ldots,0,1)$ and~$z_\xi$ is the origin. 

Let~$f_0$ be a smooth function with unit integral supported in the unit ball. Consider its dilations~$f_n$:
\eq{
f_n(x) = n^d f(nx),\qquad x\in \R^d,\ n\in \N.
}
We will be using the homogeneity relation
\eq{\label{HomKer}
K*f_n(x) = n^{d-\alpha}K*f(nx),\qquad x\in \R^d,
}
and the asymptotic formula
\eq{\label{AsymptoticFormula}
K*f(x) = K(x) + O(|x|^{\alpha - d - 1}),\qquad x\to \infty.
}
This formula follows from the homogeneity and smoothness assumptions on the kernel~$K$. Recall Lemma~$6.5$ in~\cite{Stolyarov2021bis}:
\eq{\label{L65}
\big|\Phi(a+b) - \Phi(a)\big| \lesssim |a|^{p-1}|b|,\qquad\text{provided}\quad 2|b|\leq |a|,\quad a,b\in\R^\ell.
}
Then, by~\eqref{AsymptoticFormula} and~\eqref{L65},
\eq{\label{eq24}
\Phi(K*f(x)) = \Phi(K(x)) + O\Big(|K(x)|^{p-1}|x|^{\alpha - d- 1}\Big) = \Phi(K(x)) + O\Big(|x|^{-d- 1}\Big),\qquad x\to \infty.
}
We plug the function~$x\mapsto f_n(x_1,x_2,\ldots, x_{d-1}, x_d-2/n)$ into~\eqref{BoundedInequality}; note that such a function is supported in~$\Omega$ provided~$n$ is sufficiently large. Then, by~\eqref{HomKer}, the left hand side of~\eqref{BoundedInequality} equals
\eq{\label{eq25}
\Big|\int\limits_{n\Omega}\Phi(K*f(x_1,x_2,\ldots,x_{d-1},x_d - 2))\,dx\Big|.
}
Let~$B_r(x)$ be the open Euclidean ball of radius~$r$ centered at~$x$. Note that~$K*f$ is a continuous function. Therefore, by the asymptotic formula~\eqref{eq24}, the integral~\eqref{eq25} differs by a uniformly (w.r.t.~$n$) bounded quantity from
\eq{\label{eq26}
\Big|\int\limits_{n\Omega\setminus B_4(0)}\Phi(K(x))\,dx\Big|.
}
Now let
\eq{
\I= \int\limits_{\genfrac{}{}{0pt}{-2}{\zeta \in S^{d-1}}{\zeta_d > 0}}\Phi(K(\zeta))\,d\sigma(\zeta);\qquad \M=\int\limits_{\zeta \in S^{d-1}}\Big|\Phi(K(\zeta))\Big|\,d\sigma(\zeta).
}
We wish to show~$\I = 0$, this will prove the necessity of the first identity in~\eqref{CancellationNew} (the necessity of the second is obtained by plugging a similar function generated by~$f$ whose integral is~$-1$). We rewrite~\eqref{eq26} using a change of variables and homogeneity:
\eq{\label{eq29}
\int\limits_{n\Omega\setminus B_4(0)}\Phi(K(x))\,dx = \int\limits_4^\infty r^{d-1}\!\!\!\int\limits_{S^{d-1}\cap\frac{n}{r}\Omega}\Phi(K(r\zeta))\,d\sigma(\zeta)\,dr
= \int\limits_4^\infty r^{-1}\!\!\!\int\limits_{S^{d-1}\cap\frac{n}{r}\Omega}\Phi(K(\zeta))\,d\sigma(\zeta)\,dr.
}
Note that, in fact, the domain of integration for the outer integral is bounded by~$n\diam \Omega$, we will use this fact slightly later. Let
\eq{
\Psi(\rho) =\!\!\!\! \int\limits_{S^{d-1}\cap\rho^{-1}\Omega}\!\!\!\Phi(K(\zeta))\,d\sigma(\zeta),\qquad \rho \in (0,\infty).
}
This function satisfies the bounds
\alg{
\label{eq211}|\Psi(\rho)|\leq \M;\\
\label{eq212}\Psi(\rho) = \I + O(\rho^\beta),
}
for all~$\rho \in (0,\infty)$; the second formula follows from~\eqref{Holder}. Let~$\gamma \in (0,1)$ be a parameter. Then,~\eqref{eq29} equals
\eq{
\int\limits_{4}^{\infty}r^{-1}\Psi(r/n)\,dr = \int\limits_{4}^{n^\gamma}r^{-1}\Psi(r/n)\,dr + \int\limits_{n^\gamma}^{O(n)}r^{-1}\Psi(r/n)\,dr.
}
The absolute value of the second integral is bounded by~$\M(1-\gamma)\log n + O(1)$ by~\eqref{eq211}, whereas the first equals~$\I\gamma \log n + O(1+n^{(\gamma - 1)\beta})$ via~\eqref{eq212}.
If~$\I\ne 0$, we may choose~$\gamma$ such that the whole integral tends to infinity as~$n\to \infty$, which contradicts the initial inequality. We have proved necessity in Theorem~\ref{BoundedDomainTheorem}.

To prove necessity in Theorem~\ref{HalfSpaceTheorem}, we follow a similar route. Without loss of generality, we may assume~$\xi = (0,0,\ldots,1)$ and plug the function
\eq{
x\mapsto f_n(x_1,x_2,\ldots,x_{d-1},x_d - 1/n) - f(x_1,x_2,\ldots, x_{d-1},x_d - 1)
}
into~\eqref{HalfSpaceInequality}. We need to subtract the second term to obtain a function with zero mean. Applying the tricks we have used to pass from~\eqref{BoundedInequality} to the uniform boundedness of~\eqref{eq26} to remove the contribution of the second term to the left hand side of the inequality, we arrive at the integral
\eq{
\Big|\int\limits_{\genfrac{}{}{0pt}{-2}{|x|< n}{x_d > 0}}\Phi(K(x))\,dx\Big|.
}
The uniform boundedness (w.r.t~$n$) of this integral leads to~\eqref{CancellationNew} via the same change of variable.

The necessity of~\eqref{CancellationOld} in Theorem~\ref{HalfSpaceTheorem} is proved by considering the function
\eq{
x\mapsto f_n(x_1,x_2,\ldots,x_{d-1},x_d - 1) - f(x_1,x_2,\ldots, x_{d-1},x_d - 1).
}

\section{Preliminary estimates}\label{S3}
We split the kernel into parts:
\eq{
K_n(x) =
\begin{cases}
K(x),\qquad &|x|\in [2^{-n-1},2^{-n});\\
0,\qquad &\text{otherwise}.
\end{cases}
}
Then,~$K = \sum_n K_n$ and~$K_n(x) = 2^{(d-\alpha)n} K_0(2^n x)$. We will also use the notation
\eq{
K_{\leq n}(x) = \sum\limits_{k\leq n}K_k(x).
}
The condition~\eqref{CancellationOld} implies
\eq{\label{CancellationOld2}
\int\limits_{\R^d}\Phi(K_n(x))\,dx= \int\limits_{\R^d}\Phi(-K_n(x))\,dx = 0,
}
while~\eqref{CancellationNew} yields
\eq{\label{CancellationNew2}
\int\limits_{\scalprod{x}{\xi} \geq 0}\!\!\!\Phi(K_n(x))\,dx =\!\!\!  \int\limits_{\scalprod{x}{\xi} \geq 0}\!\!\!\Phi(-K_n(x))\,dx = 0.
}
The following lemma is a simpler variation of Lemma~$2.4$ in~\cite{Stolyarov2021bis}.
\begin{Le}\label{LowFrequency}
Let~$p\in [1,\infty)$. If~$\Omega$ is a bounded domain, then
\eq{
\int\limits_\Omega|K_{\leq 0}*f(x)|^p\,dx \lesssim \|f\|_{L_1(\R^d)}^p,\qquad f\in L_1(\R^d).
}
\end{Le}
\begin{proof}
The inequality follows from
\eq{
\|K_{\leq 0}*f\|_{L_\infty} \leq \|K_{\leq 0}\|_{L_\infty} \|f\|_{L_1}\lesssim \|f\|_{L_1},
}
and the boundedness of~$\Omega$.
\end{proof}
The function~$\Me_p\colon \R_+\times \R_+\to \R$ originated in~\cite{Stolyarov2021} in this context and played the pivotal role in~\cite{Stolyarov2021bis} (see formulas~$(2.17)$ and~$(5.4)$ of that paper):
\eq{
\Me_p(x,y) = 
\begin{cases}
\min(x^{p-1}y,xy^{p-1}),\qquad &p\in (1,2];\\
\frac12(x^{p-1}y + xy^{p-1}),\qquad &p\in (2,\infty),
\end{cases}
\qquad x \geq 0,y \geq 0.
}
The lemma below is completely similar to Lemma~$2.6$ in~\cite{Stolyarov2021bis}, we omit its proof.
\begin{Le}\label{Lemma32}
The inequality 
\eq{
\Big|\int\limits_{\Omega}\Phi(K_{\leq n+1}*f) - \Phi(K_{\leq n}*f)\Big| \lesssim \Big|\int\limits_{\Omega}\Phi(K_{n+1}*f)\Big| + \int\limits_{\Omega}\Me_p\Big(|K_{\leq n}*f|,|K_{n+1}*f|\Big)
}
holds true for any bounded compactly supported~$f$ and~$n\in\N$.
\end{Le}
Since~$\Me_p$ is a non-negative function, the estimate
\eq{\label{MpEstimate}
\sum\limits_{n\in\Z}\int\limits_{\Omega}\Me_p\Big(|K_{\leq n}*f|,|K_{n+1}*f|\Big) \lesssim \|f\|_{L_1(\R^d)}^p
}
follows from Theorem~$2.3$ in~\cite{Stolyarov2021bis}. A companion estimate cannot be reduced to Theorem~$2.2$ in~\cite{Stolyarov2021bis} since now we need more cancellation conditions on the kernel~$K$.
\begin{Th}\label{BesovTheorem}
Assume~$\Omega$ is a bounded domain with~$C^{1,\beta}$-smooth boundary and~\eqref{CancellationNew} holds true. Then,
\eq{
\sum\limits_{n\in \Z}\Big|\int\limits_\Omega \Phi(K_n*f(x))\,dx\Big|\lesssim \|f\|_{L_1}^p
}
for any compactly supported bounded~$f$.
\end{Th}

\section{Sufficiency: proof of Theorem~\ref{BesovTheorem}}\label{S4}
Let~$\{Q_{k,j}\}$ be the grid of dyadic cubes:
\eq{
Q_{k,j} = \prod_{i=1}^d\big[2^{-k} j_i, 2^{-k}(j_{i} + 1)\big),\qquad j=(j_1,j_2,\ldots,j_d)\in \Z^d, \ k\in \Z.
}
If~$Q$ is a cube,~$c(Q)$ denotes its center,~$\ell(Q)$ its sidelength, and~$\lambda Q$ is the cube with center~$c(Q)$ and sidelength~$\lambda \ell(Q)$,~$\lambda > 0$. 
\begin{Def}
We say that a cube~$Q$ is a boundary cube, provided~$(d+2)Q\cap\partial \Omega \ne \varnothing$. The collection of all boundary cubes is denoted by~$\B$.
\end{Def}
\begin{Le}\label{MainLemma}
Let~$Q$ be a dyadic cube of generation~$n$,~$Q\notin\B$. Let~$f$ be a summable function supported in~$Q$. Then,
\eq{\label{eq42}
\Big|\int\limits_\Omega \Phi(K_n*f(x))\,dx\Big| \lesssim 2^n \|f\|_{L_1(Q)}^{p-1}\inf\limits_{c\in Q}\int\limits_{Q}|x-c||f(x)|\,dx.
}
\end{Le}
\begin{proof}
The function~$K_n*f$ is supported in~$(d+2)Q$. Since~$Q\notin\B$, either~$(d+2)Q\cap \Omega = \varnothing$, or~$(d+2)Q\subset \Omega$. In the first case, the integral on the left hand side of~\eqref{eq42} vanishes. In the other case,
\eq{
\int\limits_\Omega \Phi(K_n*f(x))\,dx = \int\limits_{\R^d} \Phi(K_n*f(x))\,dx.
}
The desired inequality follows from Corollary~$3.2$ in~\cite{Stolyarov2021bis} and the observation that
\eq{
\inf\limits_{c\in Q}\int\limits_{\R^d}|x-c||f(x)|\,dx = \inf\limits_{c\in \R^d}\int\limits_{\R^d}|x-c||f(x)|\,dx
}
since~$\supp f \subset Q$.
\end{proof}
\begin{Le}\label{MainBoundaryLemma}
Let~$Q$ be a boundary dyadic cube of generation~$n \geq 0$. Let~$f$ be supported in~$Q$. Then,
\eq{
\Big|\int\limits_\Omega \Phi(K_n*f(x))\,dx\Big| \lesssim 2^n \|f\|_{L_1(Q)}^{p-1}\inf\limits_{c\in \partial \Omega}\int\limits_{Q}|x-c||f(x)|\,dx + 2^{-\beta n}\|f\|_{L_1(Q)}^p.
}
\end{Le}
\begin{proof}
Let~$y\in \partial\Omega$ be an arbitrary point. Then, 
\mlt{
\Big|\int\limits_{\Omega}\Phi(K_n*f(x))\,dx\Big| \\ 
\leq \bigg|\int\limits_{\Omega}\Phi(K_n*f(x))\,dx - \int\limits_{\Omega}\Phi\Big(K_n(x-y)\cdot \int f\Big)\,dx\bigg| + \bigg|\int\limits_{\Omega}\Phi\Big(K_n(x-y)\cdot \int f\Big)\,dx\bigg|.
} 
For specific~$y\in \partial \Omega$, the first summand is bounded by~$2^n \|f\|_{L_1(Q)}^{p-1}\inf_{c\in \partial \Omega}\int_{Q}|x-c||f(x)|\,dx$ similar to the proof of Lemma~$3.1$ in~\cite{Stolyarov2021bis} and we only need to show that
\eq{
\sup\limits_{y\in \partial\Omega} \Big|\int\limits_{\Omega}\Phi(K_n(x-y))\,dx\Big| \lesssim 2^{-\beta n}.
}
Fix~$y\in \partial\Omega$ and introduce~$\xi = \norm_y$ --- the inward pointing normal to~$\partial \Omega$ at the point~$y$. Then,
\eq{
\Big|\int\limits_{\Omega}\Phi(K_n(x-y))\,dx\Big| \leq  \Big|\!\!\!\int\limits_{\genfrac{}{}{0pt}{-2}{x\colon\! \scalprod{x-y}{\xi}}{ \geq 0}}\!\!\!\!\!\!\Phi(K_n(x-y))\,dx\Big|+\!\!\!\!\!\! \int\limits_{\genfrac{}{}{0pt}{-2}{x\notin\Omega,}{\scalprod{x-y}{\xi} \geq 0}}\!\!\!|K_n(x-y)|^p\,dx+\!\!\!\!\!\! \int\limits_{\genfrac{}{}{0pt}{-2}{x\in\Omega,}{\scalprod{x-y}{\xi} \leq 0}}\!\!\!|K_n(x-y)|^p\,dx.
}
The first summand is zero by~\eqref{CancellationNew2}. To bound the second and the third summands, we use~\eqref{Holder} and note that the volume of the domains of integration is~$O(2^{-(d+\beta)n})$, while the integrands do not exceed~$2^{p(d-\alpha)n}$ in absolute value. Therefore, the second and the third terms are~$O(2^{-\beta n})$ as desired.
\end{proof}
We formulate an analog of Theorem~$3.1$ in~\cite{Stolyarov2021bis}.
\begin{Th}\label{Th41}
For any~$n \geq 0$, the inequality
\mlt{\label{LongFormula}
\Big|\int\limits_\Omega \Phi(K_{n+1}*f)\Big|\lesssim 2^n\!\!\!\sum\limits_{\genfrac{}{}{0pt}{-2}{j\in\Z^d}{3Q_{n,j}\notin\B}}\!\!\!\|f\|_{L_1(3Q_{n,j})}^{p-1}\inf\limits_{c_{n,j}}\!\!\!\int\limits_{3Q_{n,j}}\!\!\!|x-c_{n,j}||f(x)|\,dx\\ + 2^n\!\!\!\sum\limits_{\genfrac{}{}{0pt}{-2}{j\in\Z^d}{3Q_{n,j}\in\B}}\!\!\!\|f\|_{L_1(3Q_{n,j})}^{p-1}\inf\limits_{c_{n,j}\in\partial\Omega}\!\!\!\int\limits_{3Q_{n,j}}\!\!\!|x-c_{n,j}||f(x)|\,dx + 2^{-\beta n}\!\!\!\sum\limits_{\genfrac{}{}{0pt}{-2}{j\in\Z^d}{3Q_{n,j}\in\B}}\!\!\!\|f\|_{L_1(3Q_{n,j})}^p
}
holds true.
\end{Th}
Theorem~\ref{Th41} is derived from Lemmas~\ref{MainLemma} and~\ref{MainBoundaryLemma} in the same way as Theorem~$3.1$ in~\cite{Stolyarov2021bis} is deduced from Lemma~$3.1$ of that paper. We will provide a slight simplification of the argument.
\begin{proof}[Proof of Theorem~\ref{Th41}]
Recall that~$K_{n+1}$ is supported in the ball of radius~$2^{-n-1}$. We split the whole space into the strips
\eq{
\Pi_j = \Set{x\in\R^d}{x_1 \in [2^{-n}j,2^{-n}(j+1))},\qquad j\in \Z.
}
Let~$f_j = f\chi_{\Pi_j}$. We observe a useful pointwise identity
\eq{
\Phi(K_{n+1}*f) = \sum\limits_{j\in\Z}\Phi(K_{n+1}*(f_j+f_{j+1})) - \sum\limits_{j\in\Z}\Phi(K_{n+1}*f_j),
}
which leads to the bound
\eq{
\Big|\int\limits_{\Omega} \Phi(K_{n+1}*f)\Big| \leq \sum\limits_{j\in\Z}\Big|\int\limits_\Omega \Phi(K_{n+1}*(f_j+f_{j+1}))\Big| + \sum\limits_{j\in\Z}\Big|\int\limits_\Omega \Phi(K_{n+1}*f_j)\Big|.
}
Performing a similar splitting procedure for the other~$(d-1)$ coordinates, we bound the integral on the left hand side of~\eqref{LongFormula} with 
\eq{
\sum\Big|\int\limits_\Omega \Phi(K_{n+1}*\tilde{f})\Big|,
}
where each function~$\tilde{f}$ is supported in a cube~$3 Q_{n,j}$ for some~$j$, and each cube~$3Q_{n,j}$ is chosen by at most~$2^d$ functions~$\tilde{f}$. We also note that each~$\tilde{f}$ is nothing but~$f$ multiplied by an indicator. Thus, the desired bound indeed follows from Lemmas~\ref{MainLemma} and~\ref{MainBoundaryLemma}.
\end{proof}

Thus, the proof of Theorem~\ref{BesovTheorem} is naturally split into three estimates:
\alg{
\label{eq410}\sum\limits_{n \geq 0}2^n\!\!\!\sum\limits_{\genfrac{}{}{0pt}{-2}{j\in\Z^d}{3Q_{n,j}\notin\B}}\!\!\!\|f\|_{L_1(3Q_{n,j})}^{p-1}\inf\limits_{c_{n,j}}\!\!\!\int\limits_{3Q_{n,j}}\!\!\!|x-c_{n,j}||f(x)|\,dx \lesssim \|f\|_{L_1(\R^d)}^p;\\
\label{eq411}\sum\limits_{n\geq 0}2^n\!\!\!\sum\limits_{\genfrac{}{}{0pt}{-2}{j\in\Z^d}{3Q_{n,j}\in\B}}\!\!\!\|f\|_{L_1(3Q_{n,j})}^{p-1}\inf\limits_{c_{n,j}\in\partial\Omega}\!\!\!\int\limits_{3Q_{n,j}}\!\!\!|x-c_{n,j}||f(x)|\,dx \lesssim \|f\|_{L_1(\R^d)}^p;\\
\label{eq412}\sum\limits_{n\geq 0}2^{-\beta n}\!\!\!\sum\limits_{\genfrac{}{}{0pt}{-2}{j\in\Z^d}{3Q_{n,j}\in\B}}\!\!\!\|f\|_{L_1(3Q_{n,j})}^p \lesssim \|f\|_{L_1(\R^d)}^p.
}
The estimate~\eqref{eq410} was proved in~\cite{Stolyarov2021bis}, see formula~$(4.11)$ of that paper. The proof of~\eqref{eq412} is relatively simple:
\eq{
\sum\limits_{n\geq 0}2^{-\beta n}\!\!\!\sum\limits_{\genfrac{}{}{0pt}{-2}{j\in\Z^d}{3Q_{n,j}\in\B}}\!\!\!\|f\|_{L_1(3Q_{n,j})}^p \lesssim \sum\limits_{n\geq 0}2^{-\beta n}\|f\|_{L_1(\R^d)}^p \lesssim \|f\|_{L_1(\R^d)}^p.
}
It remains to justify~\eqref{eq411}. 
Let~$Q$ be a cube. By~$\D_m(Q)$ we mean the collection of all dyadic subcubes of generation~$m$ of the cube~$Q$ ($Q$ itself is of generation~$0$). Consider the quantity
\eq{
\Eb_{Q,m}[f] = \sum\limits_{\genfrac{}{}{0pt}{-2}{Q'\in \B}{Q'\in \D_m(Q)}}\Big(\int\limits_{Q'}|f(x)|\,dx\Big)^p,
}
which is the boundary modification of the quantity
\eq{
\E_{Q,m}[f] = \sum\limits_{Q'\in \D_m(Q)}\Big(\int\limits_{Q'}|f(x)|\,dx\Big)^p,
}
which played an important role in~\cite{Stolyarov2021bis}. Note that~$\Eb_{Q',1}[f] \leq \Eb_{Q',0}[f]$ for any cube~$Q'$; here we use that a parent of a boundary cube is also a boundary cube. In particular,~$\Eb_{Q,m+1}[f]\leq \Eb_{Q,m}[f]$ for any~$m \geq 0$. Therefore,
\eq{\label{Telescopic}
\sum\limits_{m \geq 0}\big(\Eb_{Q,m}[f] - \Eb_{Q,m+1}[f]\big) \leq \|f\|_{L_1(Q)}^p,
}
and all the summands in this sum are non-negative.
\begin{Le}\label{SecondCore}
Fix some~$\eps \in (0,1/2)$. Let~$Q$ be a boundary cube. Then,
\eq{
\frac{\|f\|_{L_1(Q)}^{p-1}}{\ell(Q)}\inf\limits_{c\in\partial\Omega}\int\limits_{Q}|x-c||f(x)|\,dx \lesssim \sum\limits_{m \geq 0}(1-\eps)^{m}\Big(\Eb_{Q,m}[f] - \Eb_{Q,m+1}[f]\Big).
}
\end{Le}
This lemma is similar to Lemma~$4.2$ in~\cite{Stolyarov2021bis}, the proof is also similar. We provide it for completeness. Lemma~$6.6$ from~\cite{Stolyarov2021bis} says 
\eq{\label{OldSimpleLemma}
\Big(\sum\limits_{j=1}^n z_j\Big)^p - \sum\limits_{j=1}^nz_j^p \gtrsim \Big(\sum\limits_{j=1}^n z_j\Big)^{p-1}\min\limits_{i}\sum\limits_{j\ne i}z_j
}
for any collection of non-negative numbers~$z_j$. We will need a slight modification.
\begin{Le}\label{NewSimple}
Let~$z_1,z_2,\ldots,z_n$, and~$Z$ be non-negative numbers. Then,
\eq{\label{NewSimpleLemma}
\Big(Z+\sum\limits_{j=1}^n z_j\Big)^p - \sum\limits_{j=1}^nz_j^p \gtrsim \Big(Z+\sum\limits_{j=1}^n z_j\Big)^{p-1}\Big(Z+\min\limits_{i}\sum\limits_{j\ne i}z_j\Big).
}
\end{Le}
\begin{proof}
Without loss of generality, let~$z_1$ be the largest of the~$z_j$. Consider two cases:~$z_1 \geq Z$ and~$Z > z_1$.

The first case follows from~\eqref{OldSimpleLemma} (applied to all the numbers~$z_1,z_2,\ldots,z_n$, and~$Z$):
\eq{
\Big(Z+\sum\limits_{j=1}^n z_j\Big)^{p-1}\Big(Z+\min\limits_{i}\sum\limits_{j\ne i}z_j\Big) \lesssim \Big(Z+\sum\limits_{j=1}^n z_j\Big)^p - \sum\limits_{j=1}^nz_j^p - Z^p \leq\Big(Z+\sum\limits_{j=1}^n z_j\Big)^p - \sum\limits_{j=1}^nz_j^p.
}

For the second case, we may, having in mind the homogeneity of the inequality in question, assume without loss of generality that~$Z = 1$ and, therefore~$z_j< 1$ for~$j=1,2,\ldots,n$. We note that in such a case the left hand side of~\eqref{NewSimpleLemma} is at least~$1$ (by H\"older's inequalty), whereas the right hand side does not exceed~$(n+1)^p$.
\end{proof}
\begin{proof}[Proof of Lemma~\ref{SecondCore}]
We will construct the cubes~$R_m$ starting with~$R_0 = Q$. Assume~$R_0, R_1,\ldots, R_{m}$ are already constructed. Then, we define~$R_{m+1}$ by the rules
\eq{
R_{m+1}\in \D_{1}(R_m),\ R_{m+1}\in \B,\quad \int\limits_{R_{m+1}}|f(x)|\,dx  = \max\limits_{R\in \D_1(R_m)\cap\B}\int\limits_R|f(x)|\,dx.
}
In other words, if the boundary cube~$R_m$ has kids that are also boundary cubes, we choose the one among them that has the maximal possible mass. Note that there may be no boundary cubes in~$\D_1(R_m)$. In such a case, the process terminates (this means, in particular, that the initial cube~$Q$ does not intersect~$\Omega$). Let~$c_0$ be the common point of the cubes~$R_m$ if there are infinitely many of them, or the closest point of~$\partial \Omega$ to the smallest of the cubes if there is only a finite number of cubes. Then,
\eq{\label{BoundaryDistance}
\forall m\ \forall x\in R_m\qquad |x-c_0|\lesssim 2^{-m}\ell(Q).
}

Pick a small number~$\delta$ to be chosen later. Now let~$M$ be the smallest possible non-negative integer number~$m$ such that
\eq{
\int\limits_{R_{m+1}}|f(x)|\,dx < (1-\delta)\int\limits_{R_{m}}|f(x)|\,dx.
}
If such a number~$m$ does not exist, then we set~$M=\infty$. We estimate
\eq{\label{eq428}
\int\limits_{Q}\frac{|x-c_0|}{\ell(Q)}|f(x)|\,dx \Lseqref{BoundaryDistance}\!\!\! \int\limits_{R_{M+1}}2^{-M-1}|f(x)|\,dx + \sum\limits_{m=0}^{M}\int\limits_{R_{m}\setminus R_{m+1}}2^{-m}|f(x)|\,dx.
}
The choice of~$M$ allows to disregard the first summand since it is bounded by the second one (the multiplicative constant is~$O(1/\delta)$, however, this does not harm since we will fix~$\delta$ momentarily). By the choice of~$M$,
\eq{\label{eq429}
\|f\|_{L_1(Q)} \leq (1-\delta)^{-m} \|f\|_{L_1(R_m)},\qquad m=0,1,\ldots, M.
}
Thus,
\eq{\label{SomeEquation}
\frac{\|f\|_{L_1(Q)}^{p-1}}{\ell(Q)}\inf\limits_{c\in\partial\Omega}\int\limits_{Q}|x-c||f(x)|\,dx \LseqrefTwo{eq428}{eq429} \sum\limits_{m=0}^M(1-\delta)^{-m(p-1)}\|f\|_{L_1(R_m)}^{p-1}\!\!\!\int\limits_{R_{m}\setminus R_{m+1}}\!\!\!2^{-m}|f(x)|\,dx.
}
Let~$R'_1,R'_2,\ldots,R'_n$ be the kids of~$R_m$ that are boundary cubes; let~$R_{m+1} = R'_1$. We apply Lemma~\ref{NewSimple} with with~$Z = \int_{R_{m}\setminus \cup_j R'_j}|f|$, and~$z_j = \int_{R'_j}|f|$.  Then, the right hand side in~\eqref{SomeEquation} is bounded by
\eq{
\sum\limits_{m=0}^M 2^{-m}(1-\delta)^{-m(p-1)} \big(\Eb_{R_m,0}[f] - \Eb_{R_{m},1}[f]\big)\leq
 \sum\limits_{m \geq 0}(1-\eps)^m\Big(\Eb_{Q,m}[f] - \Eb_{Q,m+1}[f]\Big),
}
provided~$(1-\eps) \geq \frac12 (1-\delta)^{-(p-1)}$ (this defines the choice of~$\delta$).
\end{proof}
\begin{proof}[Proof of Theorem~\ref{BesovTheorem}.]
By Theorem~\ref{Th41}, it suffices to justify~\eqref{eq410}, \eqref{eq411}, and~\eqref{eq412}, of which only~\eqref{eq411} remains unverified. Without loss of generality, we may assume~$\Omega$ lies inside a dyadic cube~$Q$ of generation~$0$. The three lattice theorem (see the explanation right after formula~$(4.11)$ in~\cite{Stolyarov2021bis}) says we may replace the cubes~$3Q_{n,j}$ with simply~$Q_{n,j}$ in our estimates. We apply Lemma~\ref{SecondCore} to each of the dyadic cubes on the left hand side of~\eqref{eq411} and obtain
\mlt{
\sum\limits_{n\geq 0}2^n\!\!\!\sum\limits_{\genfrac{}{}{0pt}{-2}{j\in\Z^d}{Q_{n,j}\in\B}}\!\!\!\|f\|_{L_1(Q_{n,j})}^{p-1}\inf\limits_{c_{n,j}\in\partial\Omega}\!\!\!\int\limits_{Q_{n,j}}\!\!\!|x-c_{n,j}||f(x)|\,dx\\
\lesssim \sum\limits_{n\geq 0}\!\!\!\sum\limits_{\genfrac{}{}{0pt}{-2}{j\in\Z^d}{Q_{n,j}\in\B}}\!\!\! \sum\limits_{m \geq 0}(1-\eps)^m\big(\Eb_{Q_{n,j},m}[f] - \Eb_{Q_{n,j},m+1}[f]\big) = \\
\sum\limits_{k \geq 0}\Big(\sum\limits_{l \leq k}(1-\eps)^l\Big)\big(\Eb_{Q,k}[f] - \Eb_{Q,k+1}[f]\big)\Lseqref{Telescopic}\|f\|_{L_1(Q)}^p.
}
\end{proof}

\section{End of proof and concluding remarks}\label{S5}
\begin{proof}[Proof of Theorem~\ref{BoundedDomainTheorem}]
We have already proved necessity of~\eqref{CancellationNew} (see Section~\ref{S2}). Assume now~\eqref{CancellationNew} holds. Then,~\eqref{CancellationOld} holds as well. We wish to prove~\eqref{BoundedInequality}. Fix some compactly supported bounded function~$f$. We start with the bound
\eq{\label{Telescopic}
\Big|\int\limits_\Omega \Phi(K*f)\Big|\leq \sum\limits_{n \geq 0}\Big|\int\limits_\Omega \Phi(K_{\leq n+1}*f) - \Phi(K_{\leq n}*f)\Big| + \Big|\int\limits_\Omega\Phi(K_{\leq 0}*f)\Big|,
}
which is true since~$\Phi(K_{\leq n+1}*f)\to\Phi(K*f)$ pointwise, and this sequence of functions is uniformly bounded (recall~$f$ is bounded and has compact support). The second summand on the right hand side of the above inequality was estimated in Lemma~\ref{LowFrequency}. The series in the first summand is bounded with the help of Lemma~\ref{Lemma32}, inequality~\eqref{MpEstimate}, and Theorem~\ref{BesovTheorem}.
\end{proof}

\begin{proof}[Proof of Theorem~\ref{HalfSpaceTheorem}]
We have already proved necessity of~\eqref{CancellationNew} (see Section~\ref{S2}). Let us prove sufficiency. Without loss of generality,~$\xi = (0,0,,\ldots,0,1)$,~\eqref{CancellationNew} holds true with this particular~$\xi$ and~\eqref{CancellationOld} is true as well. By dilation invariance of the problem, we may also assume that~$f$ is supported in the unit ball. We apply the same telescopic summation trick~\eqref{Telescopic}. 

In this case, we need the condition~$\int f = 0$ to bound the second term in~\eqref{Telescopic}; this is done with the help of Lemma~$2.4$ in~\cite{Stolyarov2021bis}. To bound the series, we apply Lemma~\ref{Lemma32}, use~\eqref{MpEstimate}, and only need to justify a version of Theorem~\ref{BesovTheorem} for the case where~$\Omega$ is the halfspace. Specifically, we need to prove
\eq{
\sum\limits_{n\in \Z}\Big|\int\limits_{x_d > 0} \Phi(K_n*f(x))\,dx\Big|\lesssim \|f\|_{L_1(\R^d)}^p.
}
In fact, the proof in this case is the same as in the case of bounded~$\Omega$ (recall~$f$ is supported in the unit ball and all the functions~$f*K_n$ are supported in the ball of radius two centered at the origin), however, in Lemma~\ref{MainBoundaryLemma}, we will use~\eqref{CancellationNew} only for~$\xi = (0,0,\ldots,0,1)$ since there are no other normal vectors to the boundary.
\end{proof}

\begin{proof}[Proof of Corollary~\ref{HarmonicHalfspaceCorollary}.]
Without loss of generality, let~$\xi = (0,0,\ldots,0,1)$. Consider the symmetric extension of the function~$u$:
\eq{
\tilde{u}(x) = \begin{cases}
u(x),\quad &x_d \geq 0;\\
u(-x),\quad &x_d < 0.
\end{cases}
}
The function~$\tilde{u}$ is then compactly supported and its distributional gradient satisfies 
\eq{
\big\|\Delta\tilde{u}\big\|_{\MM} \leq 2\Big(\int\limits_{\Omega}\big|\Delta u(x)\big|\,dx+ \int\limits_{\R^{d-1}}\Big|\frac{\partial u}{\partial x_d}(x_1,x_2,\ldots,x_{d-1},0)\Big|\,dx_1\,dx_2\ldots dx_{d-1}\Big). 
}
The norm~$\|\cdot\|_{\MM}$ is the total variation of a (signed) measure. We apply Theorem~\ref{HalfSpaceTheorem} to the function~$\Delta \tilde{u}$ in the role of~$f$ (the details about substituting a measure instead of an~$L_1$ function into Theorem~\ref{HalfSpaceTheorem} are left to the reader; note that here~$\nabla \tilde{u}$ is a uniformly bounded function). More specifically, we use~\eqref{FormulaForGradient} and represent~$\nabla \tilde{u}$ as~$K*f$, where~$K(\zeta) = c_d \zeta$,~$\zeta \in S^{d-1}$, and~$f = \Delta \tilde{u}$. It remains to say that the cancellation condition~\eqref{CancellationNew} reduces to~\eqref{SpecificCancellation} in this case.
\end{proof}
The proof of Corollary~\ref{HarmonicCorollary} is based on several results about harmonic functions. Though they seem to be folklore, the author did not manage to find transparent references and provides some details of proofs. 
\begin{St}\label{ExtensionTheorem}
Let~$\Omega$ be a bounded subdomain of~$\R^d$ with smooth boundary. For any smooth function~$u$ on~$\bar\Omega$ there exists a continuous compactly supported function~$v\colon \R^d \to \R$ such that~$\nabla v = \nabla u$ on~$\Omega$ and
\eq{
\|\Delta v\|_{\MM}\lesssim  \|\Delta u\|_{L_1(\Omega)} + \Big\|\frac{\partial u}{\partial n}\Big\|_{L_1(\partial \Omega)};
}
the implicit multiplicative constant in the inequality above does not depend on the particular choice of~$u$.
\end{St}
\begin{Le}\label{CannyTrick}
Let~$F\colon \Omega \to \R$ be a harmonic function continuous up to the boundary, let~$f$ be its boundary values. Let~$u\colon \Omega \to \R$ be another function with the same boundary values~$f$. Then,
\eq{
\Big\|\frac{\partial F}{\partial n}\Big\|_{L_1(\partial \Omega)} \leq \|\Delta u\|_{L_1(\Omega)} + \Big\|\frac{\partial u}{\partial n}\Big\|_{L_1(\partial \Omega)}.
}
\end{Le}
\begin{proof}
Fix small~$\eps > 0$. Let~$\varphi$ be a smooth function on~$\partial \Omega$ such that~$\|\varphi\|_{L_\infty}\leq 1$ and
\eq{
\Big\|\frac{\partial F}{\partial n}\Big\|_{L_1(\partial \Omega)} \leq (1+\eps)\int\limits_{\partial\Omega} \varphi\, \frac{\partial F}{\partial n}.
}
Let~$\Phi\colon \Omega \to \R$ be the solution of the Dirichlet problem with the boundary data~$\varphi$.  Then, by Green's formula
\eq{
\int\limits_{\partial\Omega} \varphi\, \frac{\partial F}{\partial n} = - \int\limits_{\partial \Omega} f\,\frac{\partial \Phi}{\partial n} =- \int\limits_{\partial \Omega} \varphi\,\frac{\partial u}{\partial n} + \int\limits_\Omega \Delta u\, \Phi.
}
By the maximum principle,~$\|\Phi\|_{L_\infty} \leq 1$. Therefore, by the above,
\eq{
\int\limits_{\partial\Omega} \varphi\, \frac{\partial F}{\partial n} \leq \|\Delta u\|_{L_1(\Omega)} + \Big\|\frac{\partial u}{\partial n}\Big\|_{L_1(\partial \Omega)}.
}
It remains to take arbitrarily small~$\eps$.
\end{proof}
\begin{proof}[Proof of Proposition~\ref{ExtensionTheorem}.]
We consider the Dirichlet problem on the domain~$\R^d \setminus \Omega$ with the boundary data~$u|_{\partial \Omega}$ (see $\S 23$ and~$\S 26$ in~\cite{Vladimirov1971} for the correct conditions at infinity). Denote its solution by~$U$. We claim that
\eq{\label{DNBound}
\int\limits_{\partial \Omega}\Big|\frac{\partial U}{\partial n}\Big| \lesssim\|\Delta u\|_{L_1(\Omega)} + \Big\|\frac{\partial u}{\partial n}\Big\|_{L_1(\partial \Omega)}.
}
We hope that identical notation for the inward and outward pointing normal vectors does not lead  to ambiguity. Let us assume this bound for a while and construct the desired function~$v$. Let~$\Omega \subset B_{R/2}(0)$. Denote by~$\mathcal{R}$ the spherical layer~$B_{2R}(0)\setminus B_R(0)$. We will need the standard bound
\eq{\label{StandardBound}
\|\nabla U\|_{L_\infty(\mathcal{R})}\lesssim \int\limits_{\partial \Omega}\Big|\frac{\partial U}{\partial n}\Big|.
}
Let~$\Psi$ be a smooth function that equals~$1$ on~$B_R(0)$ and zero outside~$B_{2R}(0)$; set
\eq{
W(x) = \begin{cases}
U(x),\qquad &x\notin\Omega;\\
u(x),\qquad &x\in \Omega,
\end{cases}
}
and finally define~$v$ by the rule
\eq{
v(x) = (W(x) - W(x_0))\Psi(x),\qquad x\in \R^d,
}
where~$x_0$ is an arbitrary point in~$\mathcal{R}$. Then,
\eq{
\Delta v(x) = (W(x) - W(x_0))\Delta\Psi(x) + 2\scalprod{\nabla W(x)}{\nabla \Psi(x)} + \Delta W(x)\Psi(x).
}
Let us estimate the total variation of each summand individually:
\begin{enumerate}[1)]
\item 
\mlt{
\int\limits_{\R^d}|W(x) - W(x_0)||\Delta\Psi(x)|\,dx \lesssim \int\limits_{\mathcal{R}}|W(x) - W(x_0)|\,dx \lesssim \int\limits_{\mathcal{R}}|\nabla U(x)|\,dx\\ \LseqrefTwo{StandardBound}{DNBound} \|\Delta u\|_{L_1(\Omega)} + \Big\|\frac{\partial u}{\partial n}\Big\|_{L_1(\partial \Omega)};
}
\item 
\eq{
\int\limits_{\R^d}\Big|\scalprod{\nabla W(x)}{\nabla \Psi(x)}\Big|\,dx \lesssim \int\limits_{\mathcal{R}}|\nabla U(x)|\,dx \lesssim \|\Delta u\|_{L_1(\Omega)} + \Big\|\frac{\partial u}{\partial n}\Big\|_{L_1(\partial \Omega)};
}
\item 
\eq{
\big\|\Psi \Delta W\big\|_{\MM} \leq \int\limits_{\Omega}|\Delta u(x)|\,dx + \int\limits_{\partial \Omega}\Big|\frac{\partial u}{\partial n}\Big| +  \int\limits_{\partial \Omega}\Big|\frac{\partial U}{\partial n}\Big| \lesssim \|\Delta u\|_{L_1(\Omega)} + \Big\|\frac{\partial u}{\partial n}\Big\|_{L_1(\partial \Omega)}.
}
\end{enumerate}

Thus, it remains to prove~\eqref{DNBound}. Let~$\DNo$ and~$\DNi$ be the outer and inner Dirichlet-to-Neumann operators. By Subsection~$7.10$ in~\cite{Taylor1996}, both these are pseudodifferential elliptic operators of order~$1$; their principal symbols are equal to~$\sqrt{\Delta_{\partial\Omega}}$, where~$\Delta_{\partial \Omega}$ is a (positive) Laplacian on~$\partial\Omega$ (for example, the Laplace--Beltrami operator). In the light of Lemma~\ref{CannyTrick}, the inequality~\eqref{DNBound} may be restated as
\eq{
\|\DNo f\|_{L_1(\partial\Omega)}\lesssim \|\DNi f\|_{L_1(\partial \Omega)}, \quad \text{where}\ f = u|_{\partial \Omega}.
}
By the information above,~$\DNi$ admits a left parametrix, which has the principal symbol~$(\sqrt{\Delta_{\partial\Omega}})^{-1}$. By and the standard calculus of pseudodifferential operators (Subsections~$7.3$ and~$7.4$ of~\cite{Taylor1996}), this implies
\eq{
\DNo f = \DNi f + \mathrm{Er}\DNi f, 
}
where~$\mathrm{Er}$ is a pseudodifferential operator of order~$-1$; since~$\partial\Omega$ is compact, it maps~$L_1$ to itself. The inequaltiy~\eqref{DNBound} is proved.
\end{proof}
\begin{proof}[Proof of Corollary~\ref{HarmonicCorollary}.]
We apply Proposition~\ref{ExtensionTheorem} and construct the function~$v$. Then, by~\eqref{FormulaForGradient},~$\nabla u(x) = K*\Delta v (x)$ for~$x\in \Omega$, for~$K(\zeta) = c_d \zeta$,~$\zeta \in S^{d-1}$. It remains to apply Theorem~\ref{BoundedDomainTheorem}.
\end{proof}


\bibliography{/Users/mac/Documents/Bib/Mybib}{}
\bibliographystyle{plain}

St. Petersburg State University, Department of Mathematics and Computer Science;


d.m.stolyarov at spbu dot ru.
\end{document}